\documentclass[preprint,12pt]{elsarticle}

\usepackage{amssymb}
\usepackage{amsmath}
\usepackage{amsfonts}
\usepackage{graphicx}
\usepackage{graphics}
\usepackage{tikz}
\usepackage{titlesec}
\usepackage[english]{babel}

\newtheorem{theorem}{Theorem}

\newtheorem{lemma}[theorem]{Lemma}

\newtheorem{remark}[theorem]{Remark}

\def\qed{\vbox{\hrule
 \hbox{\vrule\hbox to 5pt{\vbox to 8pt{\vfil}\hfil}\vrule}\hrule}}

\usepackage{tikz}
\usepackage{tkz-graph}

\journal{   }

\begin{document}
\begin{frontmatter}

\title{Some New Lower Bounds for the Estrada Index}

\author[]{Juan L. Aguayo and Juan R. Carmona\corref{cor1}}
\ead{juanaguayo@uach.cl, juan.carmona@uach.cl}
\address{Instituto de Ciencias F\'isicas y Matem\'aticas,\\
  Universidad Austral de Chile, Independencia 631, Valdivia, Chile}
\cortext[cor1]{Corresponding author}

\author{Jonnathan Rodr\'{\i}guez}
\ead{jonnathan.rodriguez@uantof.cl}
\address{Departamento de Matem\'{a}ticas, Facultad de Ciencias B\'asicas\\
Universidad de Antofagasta, Antofagasta, Chile.}

\begin{abstract}
Let $G$ be a graph on $n$ vertices and $\lambda_1,\lambda_2,\ldots,\lambda_n$ its  eigenvalues. The Estrada index of $G$ is an invariant that is calculated from the eigenvalues of the adjacency matrix of a graph.  In this paper, we present some new lower bounds obtained for the Estrada Index of graphs and in particular of bipartite graphs  that only depend on the number of vertices, the number of edges, Randi\'c index, maximum and minimum degree and diameter.
\end{abstract}

\begin{keyword}

Estrada Index; Adjacency matrix; Lower bound; Randi\'c Index; Graph.
\MSC 05C50 \sep 15A18

\end{keyword}

\end{frontmatter}


\section{Introduction}

In this paper, we consider undirected simple graphs $G$ with by edge set denoted by $\mathcal{E}(G)$ and its vertex set $V(G)=\{1, \ldots, n \}$ with cardinality $m$ and $n$, respectively. We say that $G$ is an
$(n,m)$ graph.
	If $ e \in \mathcal{E}(G)$ has end vertices $i $ and $j$, then we say that $i$ and $j$ are adjacent and this edge is denoted by $ij$. For a finite set $U$, $|U|$ denotes its cardinality.
	Let $K_n$ be the complete graph with $n$ vertices and $\overline{K_n}$ its (edgeless) complement. A graph $G$ is bipartite if there exists a partitioning of $V(G)$ into disjoint, nonempty sets $V_1$ and $V_2$ such that the end vertices of each edge in $G$ are in distinct sets $V_1$, $V_2$. A graph $G$ is a complete bipartite graph if $G$ is bipartite and each vertex in $V_1$ is connected to all the vertices in $V_2$. If $|V_1|=p$ and $|V_2|=q$, the complete bipartite graph is denoted by $K_{p,q}$. For more properties of bipartite graphs, see \cite{haemers}.\\
	
	If $i \in V(G)$, then $NG(i)$ is the set of neighbors of the vertex $i$ in $G$, that is, $NG(i) = \{j \in V(G): ij \in \mathcal{E}(G)\}$. For the $i$-th vertex of $G$, the cardinality of $NG(i)$ is called the degree of $i$ and it is denoted by $d(i)$. The maximum vertex degree is denoted by $\Delta$ and the minimum vertex degree of $G$, is denoted by $\delta$.
	The general Randi\'c index $R_{\alpha}(G)$ of $G$, is defined as $$R_{\alpha} =\sum_{ij \in \mathcal{E}(G)} ( d(i) d(j))^{\alpha},$$ where $\alpha$ is ar arbitary real number. Si $\alpha=-1/2$ we obtain the Randi\'c index or connectivity index or branching index, denoted by $R$ (see \cite{R1}). Let $i, j \in  V(G)$, a walk of $G$ from $i$ to $j$ is a finite alternating sequence $i_0 (= i ) e_1 i_1 e_2 i_2 \ldots e_{k-1}i_{k-1} e_k i_k (= j )$ of vertices and edges such that $e_r = i _{r−1}i_{r}$ for $r = 1 , 2 ,\ldots, k $. The number $k$ is the length of the walk. In particular, if the vertex $i_r$ , $r = 0 , 1 ,\ldots, k$, in the walk are all distinct then the walk is called a path. The path graph of $n$ vertices is denoted by $P_n $. A closed path or cycle , is a path $i_1,\ldots, i_k$ (where $k \geq3$) together with the edge $i_1 i_k$ . The cycle graph of $n$ vertices is denoted by $C_n$. If each pair of vertices in a graph is joined by a walk, the graph is said to be connected,  in this case, we say that $A(G)$ is irreducible. The distance between two vertices $i$ and $j$, denoted by $d(i, j)$, is the number of edges of a shortest path between $i$ and $j$, and its maximum value over all pair of vertices is called diameter of the graph $G$, is that, $$D = diam(G) = \max\{d(i,j):i,j \in V(G)\}).$$
	The  adjacency  matrix $A(G)$ of  the  graph $G$ is  a 
	symmetric matrix of order $n$ with entries $a_{ij}$, such that $a_{ij}=1$ if $ij \in \mathcal{E}(G)$ and $a_{ij}=0$ otherwise.  
	Denoted by $\lambda_1\geq \ldots \geq \lambda_n$ to the eigenvalues of $A(G)$,  
	see \cite{C-D-S1, C-D-S2}. A matrix is singular if it has zero as an eigenvalue, otherwise, it is called non-singular. 
    

	\noindent The Estrada index of the graph $G$ is defined as
	$$EE(G)= \sum_{i=1}^n e^{\lambda_i}.$$
	This spectral quantity is put forward by Ernesto Estrada \cite{E1} in the year 2000.
	There have been found a lot of chemical and physical applications, including quantifying the degree of folding of long-chain proteins, \cite{E1, E2, E3, G-F-G-M-V, G-R-F-M-S, G-G-M-S}, 
	and complex networks \cite{E4, E5, S1, S2, S3, S4}. 
	Mathematical properties of this invariant can be found in e.g. \cite{F-A-G,G-D-R,G-R,K,S5,Z,Z-Z}. \\


	Denote by $M_k = M_k(G)$ to the $k$-th spectral moment of the graph $G$, i.e.,
	$$M_k= \displaystyle \sum_{i=1}^n(\lambda_i)^k.$$
	Then, we can write the Estrada index as
	$$EE(G) = \sum_{k=0}^{\infty}\frac{M_k}{k!}.$$
	In \cite{C-D-S1}, for an $(n, m)$-graph $G$, the authors proved that
	\begin{equation}\label{eq1}
	M_0 =n,\,\, M_1 =0,\,\, M_2 =2m,\,\, M_3 =6t,
	\end{equation}
	where $t$ is the number of triangles in $G$.\\
	

    


\noindent	In this work, under motivation of paper \cite{KarJar}, we obtain new lower bounds for the Estrada index of graphs and bipartite graphs in terms of number of vertices, number of edges, Randi\'c index, maximum and minimum degree and diameter.\\


\section{Lemmas for new lower bounds for the Estrada index}

We list some known results that will be needed in the following sections.\\

\begin{lemma}[D. Cvetkovi\'c, M. Doob, H. Sachs \cite{C-D-S1}]\label{lema1}
Let $G$ be a connected graph with $m$ edges, $n$ vertices, then
$$\lambda_1\geq \frac{2m}{n}.$$
\end{lemma}

\begin{lemma}[Bollob\'as, Erd\"os \cite{be98}]\label{lema2}
Let $G$ be a non-trivial graph with $n$ vertices, then
$$R\leq \frac n2.$$
\end{lemma}

\begin{lemma}[Favaron, Mah\'eo, J.-F. Sacl\'e \cite{fms93}]\label{lema3}
Let $G$ be a connected graph with $m$ edges, then
$$\lambda_1\geq \frac{m}{R}.$$
\end{lemma}

\begin{lemma}[Favaron, Mah\'eo, J.-F. Sacl\'e \cite{fms93}]\label{lema4}
Let G be a non-empty graph with maximum vertex degrees $\Delta$. Then 
$$\lambda_1\geq \sqrt{\Delta},$$
\end{lemma}

\begin{lemma}[Favaron, Mah\'eo, J.-F. Sacl\'e \cite{fms93}]\label{lema5}
If $G$ is a graph with $n$ vertices, $m$ edges, and degree sequence $d_1,d_2,\ldots,d_n$, then
$$\lambda_1\geq \frac{R_{1/2}}{m}$$
\end{lemma}

\begin{lemma}[Cvetkovi\'c, Doob, Sachs \cite{C-D-S1}]\label{lema6}
$G$ has only one distinct eigenvalue if and only if $G$ is an empty graph. $G$ has two distinct eigenvalues $\mu_1 > \mu_2$
with multiplicities $m_1$ and $m_2$ if and only if $G$ is the direct sum of $m_1$ complete graphs of order $\mu_1 + 1$. In this case, $\mu_2=−1$ and $m_2=m_1\mu_1$.
\end{lemma}

\begin{lemma}[Stevanovic \cite{S}]\label{lema7}
If $G$ is a connected graph with $n$ vertices and diameter $D$. Then $$\lambda_1\geq \sqrt[D]{n-1}.$$
\end{lemma}

\begin{lemma}[Das, Mojallal \cite{dm13}]\label{lema8}
If $G$ is a graph with $n$ vertices, $m$ edges, and minimum degree $\delta$. Then
$$ \lambda_1\geq\frac{2(m-\delta)}{n-1}.$$

\end{lemma}

\begin{remark}\label{obs8}

The case of equality is not discussed in \cite{dm13}. In the development of our work, we hope that a family that fulfills equality is the following
$$\Gamma=\{G \cup K_1 / G \ \text{is} \ r-\text{regular}\}.$$
Indeed, if $G=(n.m)$ we have to, $G \cup K_1$ is obtained
$$\dfrac{2(\frac{nr}{2}-0)}{n}=r=\lambda_1(G \cup K_1).$$
\end{remark}
\begin{lemma}[Hong \cite{h86}]\label{lema9}
If $G$ is a connected unicyclic graph, then
$$ \lambda_1\geq 2,$$
with equality if and only if $G$ is a cycle $C_n$.
\end{lemma}

\begin{lemma}[Collatz, Sinogowitz \cite{cs57}]\label{lema10}
If $G$ is a connected graph with $n$ vertices, then
$$ \lambda_1\geq2\cos{\left(\frac{\pi}{(n+1)} \right)} $$
with equality if and only if $G$ is isomorphic to $P_n$.
\end{lemma}


\section{Main Results}

In this section, we present our main results for the Estrada Index, which are lower bounds that are a function of known structural elements of the graph.

\subsection{\textbf{Lower bound for the Estrada index}}
\vspace{0.3cm}

Consider the following function 
\begin{equation}\label{f}
f(x)= (x-1)-\ln(x), \, \, x>0.
\end{equation}
that is is decreasing in $(0,1]$ and increasing in $[1,+\infty)$
		Thereby, $f(x) \geq f(1) = 0$, implying that 
		\begin{equation}\label{eq7}
		x \geq 1 + \ln x, \,\,\,x > 0,
		\end{equation}
		the equality holds if and only if $x = 1$. 
		Let $G$ be a graph of order $n$, using (\ref{eq1}) and (\ref{eq7}), we get:
		\begin{equation}\label{Energia0}
		\begin{array}{lllll}
		EE(G) & \geq & e^{\lambda_1} + (n-1)+ \displaystyle \sum_{k=2}^n \ln e^{\lambda_k}  \\
		& = & e^{\lambda_1} + (n -1) +\displaystyle \sum_{k=2}^n \lambda_k \\
		& = &  e^{\lambda_1} + (n -1) +M_1 - \lambda_1\\
		
		& = &  e^{\lambda_1} + (n -1)  - \lambda_1.
		\end{array}
		\end{equation}
		
\begin{remark}\label{obs}
The equality in (\ref{Energia0}) holds if and only if $\lambda_2=\ldots=\lambda_n=0$. Then, by Lemma \ref{lema6}, $G$ is isomorphic to $\overline{K_n}$.
\end{remark}
		
Define the function
		\begin{equation}\label{fi}
		\phi(x)= e^x + (n-1) - x,\qquad x\geq0.
		\end{equation}
		
\noindent Note that, this is an increasing function on $D_{\phi}=[0,+\infty)$.\\

Next, together with that the above Remark \ref{obs}, we will establish the following results based on the Lemmas mentioned in the previous section.

\begin{theorem}
Let $G$ be a connected graph with $m$ edges, then
\begin{equation}\label{teo1}
 EE(G)> e^{\left(\frac{m}{R}\right)}+(n-1)-\left(\frac{m}{R}\right),  \end{equation}
with $R=R(G)$ the Randi\'c index of $G$.
\end{theorem}
\begin{proof}
Using Lemma \ref{lema3} and (\ref{fi}), we have 
$$\phi(\lambda_1)\geq\phi\left(\frac{m}{R}\right).$$
As $G$ is connected, the equality in (\ref{teo1}) is not verified, see Remark \ref{obs}.
\end{proof}

\begin{theorem}
Let G be a non-empty graph with maximum vertex degrees $\Delta$. Then 
\begin{equation}\label{teo2}
 EE(G)> e^{\sqrt{\Delta}}+(n-1)-\sqrt{\Delta}.   
\end{equation}
\end{theorem}
\begin{proof}
Using Lemma \ref{lema4} and (\ref{fi}), we have 
$$\phi(\lambda_1)\geq\phi(\sqrt{\Delta}).$$
As $G$ is non empty, the equality in (\ref{teo2}) is not verified. (see Remark \ref{obs})
\end{proof}

\begin{theorem}
If $G$ is a graph with $n$ vertices, $m$ edges and degree sequence $d_1, d_2,\ldots,d_n$, then 
\begin{equation}\label{teo3}
 EE(G)>e^{\frac{R_{1/2}}{m}}+(n-1)-\frac{R_{1/2}}{m},    
\end{equation}
\end{theorem}
\begin{proof}
Using Lemma \ref{lema5} and (\ref{fi}), we have 
$$\phi(\lambda_1)\geq\phi\left(\frac{R_{1/2}}{m}\right).$$
As $G$ is non empty, the equality in (\ref{teo3}) is not verified, see Remark \ref{obs}.
\end{proof}

\begin{theorem}
If $G$ is a connected graph with $n$ vertices and diameter $D$. Then,
\begin{equation}\label{teo4}
 EE(G) > e^{\sqrt[D]{n-1}}+(n-1)-\sqrt[D]{n-1}.
\end{equation}
\end{theorem}
\begin{proof}
Using Lemma \ref{lema7} and (\ref{fi}), we have
$$\phi(\lambda_1)\geq\phi(\sqrt[D]{n-1}).$$
As $G$ is connected, the equality in (\ref{teo4}) is not verified. (see Remark \ref{obs})
\end{proof}

\begin{theorem}
If $G$ is a graph with $n$ vertices, $m$ edges and minimum degree $\delta$. Then
\begin{equation}
 EE(G)\geq e^{\frac{2(m-\delta)}{(n-1)}}+(n-1)-\frac{2(m-\delta)}{(n-1)}.
\end{equation}
Moreover, the equality is verified hold if and only if $G$ is isomorph to $\overline{K_n}$. 
\end{theorem}
\begin{proof}
Ussing Lemma \ref{lema8} and (\ref{fi}), we have $$\phi(\lambda_1)\geq \phi\left(\frac{2(m-\delta)}{(n-1)}\right).$$
 The equality is verified, from Remark \ref{obs}, if and only if $G$ is isomorph to $\overline{K_n}$.
\end{proof}

\begin{theorem}
If $G$ is a connected unicyclic graph, then
\begin{equation}\label{teo6}
 EE(G) > e^{2}+(n-3).
\end{equation}
\end{theorem}
\begin{proof}
Using Lemma \ref{lema9} and (\ref{fi}), we have 
$$\phi(\lambda_1)\geq\phi(2).$$
As $G$ is connected, the equality in (\ref{teo6}) is not verified. (see Remark \ref{obs})
\end{proof}

\begin{theorem}
If $G$ is a connected graph with $n$ vertices, then
\begin{equation}\label{teo7}
 EE(G) > e^{2\cos{\left(\frac{\pi}{n+1}\right)}}+(n-1)
 -2\cos{\left(\frac{\pi}{n+1}\right)}.
\end{equation}
\end{theorem}
\begin{proof}
Using Lemma \ref{lema10} and (\ref{fi}), we have 
$$\phi(\lambda_1)\geq\phi\left(2\cos\left(\frac{\pi}{n+1}\right)\right).$$
As $G$ is connected, the equality in (\ref{teo7}) is not verified. (see Remark \ref{obs})
\end{proof}
\vspace{1cm}

\subsection{\textbf{The Estrada index of a bipartite graph}}\label{bip}
\vspace{0.5cm}
In the following result, we obtain a sharp lower bound of the Estrada index for a bipartite graph. Considering (\ref{eq1}) and (\ref{eq7}), we obtain
	\begin{equation}\label{EE2}
	\begin{array}{lllll}
	EE(G) & = & e^{\lambda_1}+e^{-\lambda_1} + \displaystyle \sum_{k=2}^{n-1} e^{\lambda_k}  \\
	& \geq & 2\cosh{\lambda_1} + (n-2) + \displaystyle \sum_{k=2}^{n-1} \lambda_k \\
	& = &  2\cosh{\lambda_1} + (n-2) + M_1 + \lambda_1 - \lambda_1\\
	& = &   2\cosh{\lambda_1} +(n-2).
	\end{array}
	\end{equation}	
	Since, 
	\begin{equation}\label{fi_bi}
	    \Phi(x)= 2\cosh{x} + (n-2),
	\end{equation}
	is an increasing function on $D_{\Phi}=[0,+\infty)$.

\begin{theorem}\label{T20}
Let $G$ be a connected graph with $m$ edges, then
\begin{equation}\label{teo18}
EE(G)\geq 2\cosh{\left(\frac{m}{R}\right)} +(n-2),
\end{equation}
with $R=R(G)$ the Rand\'ic index of $G$. Moreover, the equality is verified hold if and only if $G$ is isomorph to ${K_{p,q}}$, where $p+q=n$.
\end{theorem}
\begin{proof}

Using Lemma \ref{lema4} and (\ref{fi_bi}), we have 
$$\Phi(\lambda_1)\geq\Phi\left(\dfrac mR\right).$$
So we obtain
$$EE(G)\geq 2\cosh{\left(\frac{m}{R}\right)} +(n-2).$$

Suppose that equality in (\ref{teo18}) is maintained. Then in (\ref{EE2}) inequalities are changed by equalities. Thus, by (\ref{eq7}), we have  
\begin{equation}\label{imp}
   |\lambda_1|=|\lambda_n|=\dfrac{m}{R}  \quad \text{and} \quad \lambda_2=\ldots=\lambda_{n-1}=0.
\end{equation}
From (\ref{imp}), considering to the definition of imprimitivity $h$ in (\cite{M}, Section III) and by the general form of the Frobenius matrix of a nonnegative, irreducible and symmetric matrix, we have to $A(G)$ is permutationally equivalent to a block matrix of the form
$$\left( \begin{array}{cc}
    0_{p,p} & C  \\
     C^T & {0}_{q,q}
\end{array}
\right),
$$
where $C=(c_{i,j}),  c_{i,j}=1$ for $i=1,\ldots,q.$ and $j=1,\ldots,p.$ Then $G=K_{p,q}$, with $p+q=n$.\\
On the other hand, if $G=K_{p,q}$, we have to $m=pq$ and $R=\dfrac{pq}{\sqrt{pq}}$, then:
$$2\cosh\left(\dfrac mR\right)+(p+q-2)=2\cosh(\sqrt{pq})+(p+q-2)=EE(K_{p,q}).$$
\end{proof}

\begin{theorem}\label{teo21}
Let G be a non-empty graph with $n$ vertices and maximum vertex degrees $\Delta$. Then 
\begin{equation}\label{teo19}
 EE(G)\geq 2\cosh{\sqrt{\Delta}}+(n-2),   
\end{equation}
with equality if and only if $G \cong S_n$ or $G \cong S_{\Delta+1} \cup (n-\Delta-1)K_1$.
\end{theorem}
\begin{proof}
Using Lemma \ref{lema4} and (\ref{fi_bi}), we have 
$$\Phi(\lambda_1)\geq\Phi\left(\sqrt{\Delta}\right).$$
So we obtain 
$$EE(G)\geq 2\cosh{\sqrt{\Delta}} +(n-2).$$

Suppose that equality in (\ref{teo19}) is hold. Then in (\ref{EE2}) inequalities are changed by equalities. Thus, by (\ref{eq7}), we have  $\lambda_2=\ldots=\lambda_{n-1}=0$ and $\lambda_1=-\lambda_n=\sqrt{\Delta}$.  
Since
\begin{equation*}
    \sum\limits_{i=1}^{n} \lambda^{2}_{i}=2m,
\end{equation*}
then
\begin{equation*}
    \Delta = m
\end{equation*}
By Caporossi et al in \cite{C1}, Theorem 1, we have 
\begin{equation*}
    E(G)= 2\sqrt{m}
\end{equation*}
and $G$ is isomorphic to a complete bipartite graph with isolate vertices.

We claim two cases.
\begin{itemize}
    \item \textbf{Case 1}: Supposed a connected graph $G$ then $G \cong K_{p,q}$ where $p+q=n.$
Without loss of generality, we can suppose that $p=\max\{p,q\}=\Delta.$ Moreover, we have $pq=\Delta$, thereby $q=1.$ Therefore, $G \cong S_{n}.$ 
\item \textbf{Case 2}: If $G$ is not connect graph then $G$ is isomorphic to a complete  bipartite graph with isolate vertices, i.e, $G \cong K_{p,q} \cup (n-p-q)K_1$. Then, without loss of generality, we can suppose that $p=\max\{p,q\}=\Delta.$ Moreover, we have $pq=\Delta$. Asi  $G \cong K_{1,\Delta}\cup (n-\Delta-1)K_1.$
\end{itemize}
On the other hand, if $G$ is isomorphic to the graph in Theorem, the equality in (\ref{teo19}) is easily verified.

\end{proof}

\begin{theorem}
If $G$ is a graph with $n$ vertices, $m$ edges, then 
\begin{equation}\label{teo20}
 EE(G) \geq 2\cosh{\left(\frac{R_{1/2}}{m}\right)}+(n-2).    
\end{equation}
 The equality in (\ref{teo20}) holds if and only if  $G \cong K_{p,q} \cup (n-p-q)K_1$.
\end{theorem}
\begin{proof}
Using Lemma \ref{lema5} and (\ref{fi}), we have 
$$\phi(\lambda_1)\geq \phi\left(\frac{R_{1/2}}{m}\right).$$
So we obtain

$$ EE(G) \geq 2\cosh{\left(\frac{R_{1/2}}{m}\right)}+(n-2).$$

Suppose that equality in (\ref{teo20}) is hold. Then in (\ref{EE2}) 
inequalities are changed by equalities. Thus, by (\ref{eq7}), we have  $\lambda_2=\ldots=\lambda_{n-1}=0$ and $\lambda_1=-\lambda_n=\frac{R_{1/2}}{m}$. Then the adjacency matrix of $G$ it has a index of imprimitivity $h=2$. Then, similar to the proofs in Theorem \ref{teo21}, one can easily obtain that $G \cong K_{p,q}$  the complete bipartite graph, with $p+q=n$, or  $G \cong K_{p,q}\cup (n-p-q)K_1$.

On the other hand,  if $G$ if $G \cong K_{p,q}\cup (n-p-q)K_1$, we have to $m=pq$ and $R_{1/2}={pq}{\sqrt{pq}}$, then:
$$2\cosh\left(\dfrac{R_{1/2}}{m}\right)+(n-p-q-2)=2\cosh(\sqrt{pq})+(n-p-q-2)=EE(G).$$

.\end{proof}

\begin{theorem}
If $G$ is a connected graph with $n$ vertices and diameter $D$. Then,
\begin{equation}\label{teo20}
 EE(G) \geq 2\cosh{\sqrt[D]{n-1}}+(n-2).
\end{equation}
with equality if and only if $G \cong S_n$.
\end{theorem}
\begin{proof}
Using Lemma \ref{lema8} and (\ref{fi}), we have
$$\phi(\lambda_1)\geq\phi(\sqrt[D]{n-1}).$$
So we obtain

$$ EE(G) \geq 2\cosh{\sqrt[D]{n-1}}+(n-2).$$

Suppose that equality in (\ref{teo20}) is maintained. Then in (\ref{EE2}) inequalities are changed by equalities. Thus, by (\ref{eq7}), we have  
\begin{equation}\label{imp2}
   |\lambda_1|=|\lambda_n|=\sqrt[D]{n-1}  \quad \text{and} \quad \lambda_2=\ldots=\lambda_{n-1}=0.
\end{equation}
From (\ref{imp2}), the adjacency matrix of $G$ it has a index of imprimitivity $h=2$ and similar to the proofs in Theorem \ref{teo21}, one can easily obtain that $G \cong K_{p,q}$, with $p+q=n$. Further of (\ref{imp2}) we have that
\begin{equation}\label{12345}
   \sqrt{pq}=\sqrt[D]{n-1}.
\end{equation}
For equality in (\ref{12345}) is verified, we must consider the following:

$$D=2,\quad pq=n-1\quad  \text{and}\quad p+q=n.$$ 

Note that we have a system of equations for $p$ and $q$ based on $n$. After solving the system we get the following pair of solutions $(p,q)=(1,n-1)=(n-1,1)$, therefore $G \cong S_n$.

On the other hand, if $G \cong S_n$ it is easy to check that the equality in (\ref{teo20}) holds.
\end{proof}

\begin{theorem}
If $G$ is a graph with $n$ vertices, $m$ edges and minimum degree $\delta$. Then
\begin{equation}\label{teo12}
 EE(G)\geq 2\cosh{\frac{2(m-\delta)}{(n-1)}}+(n-2).
\end{equation}
Equality is holds if and only if $G$ is isomorph to {$K_{p,p}\cup K_1$, where $n=2p+1$}.
\end{theorem}
\begin{proof}
Using Lemma \ref{lema7} and (\ref{fi}), we have
$$\phi(\lambda_1)\geq\phi\left(\frac{2(m-\delta)}{(n-1)}\right).$$
So we obtain

$$EE(G)\geq 2\cosh{\frac{2(m-\delta)}{(n-1)}}+(n-2).$$

Suppose that equality in (\ref{teo12}) is maintained. Then in (\ref{EE2}) 
inequalities are changed by equalities. Thus, by (\ref{eq7}), we have  
\begin{equation}\label{imp12}
   |\lambda_1|=|\lambda_n|= \frac{2(m-\delta)}{(n-1)} \quad \text{and} \quad \lambda_2=\ldots=\lambda_{n-1}=0.
\end{equation}
Then the adjacency matrix of $G$ it has a index of imprimitivity $h=2$. Then, similar to the proofs in the Theorem \ref{teo21}, one can easily obtain that $G \cong K_{p,q}$  the complete bipartite graph, with $p+q=n$, or  $G \cong K_{p,q}\cup (n-p-q)K_1$.
Note that the first case is discarded because it does not comply with equality
\begin{equation}\label{wwe}
    \lambda_1= \frac{2(m-\delta)}{n-1}.
\end{equation}
If $G \cong K_{p,q}\cup (n-p-q)K_1$, then $\delta=0$, and $$\lambda_1= \frac{2m}{n-1}=\lambda_1(K_{p,q}).$$ Thus we conclude that $p=q$ and $p+1=n-1$. Then $G$ is isomorph to $K_{p,p}\cup K_1$, where $n=2p+1.$
On the other hand, if $G \cong K_{p,p}\cup K_1$ it is easy to check that the equality in (\ref{teo12}) holds.
\end{proof}

\begin{theorem}
If $G$ is a connected unicyclic graph, then
\begin{equation}\label{teo13}
 EE(G) \geq 2\cosh({2})+(n-2).
\end{equation}
with equality if and only if $G \cong C_4$.
\end{theorem}

\begin{proof}
Using Lemma \ref{lema9} and Lemma \ref{fi}, we have 
$$\phi(\lambda_1)\geq\phi(2).$$
Suppose that equality in (\ref{teo13}) is maintained. Then in (\ref{EE2}) 
inequalities are changed by equalities. Thus, by (\ref{eq7}), we have  
\begin{equation}\label{imp13}
   |\lambda_1|=|\lambda_n|=2 \quad \text{and} \quad \lambda_2=\ldots=\lambda_{n-1}=0.
\end{equation}
From (\ref{imp13}), the adjacency matrix of $G$ it has a index of imprimitivity $h=2$ and similar to the proofs in Theorem \ref{T20}, one can easily obtain that $G \cong K_{p,q}$, with $p+q=n$. Further of (\ref{imp13}) y el Lemma \ref{lema9} we have that $G \cong C_n$. Then the only cycle that meets the conditions is $C_4$.
On the other hand, if $G \cong C_4$ it is easy to check that the equality in (\ref{teo13}) holds.

\end{proof}

\begin{theorem}
If $G$ is a connected graph with $n$ vertices, then
\begin{equation}\label{teo14}
 EE(G) \geq 2\cosh{\left(2\cos{\left(\frac{\pi}{n+1}\right)}\right)}+(n-2).
\end{equation}
with equality if and only if $G \cong P_2$  or  $G \cong P_4.$
\end{theorem}
\begin{proof}
Using Lemma \ref{lema10} and (\ref{fi}), we have 
$$\phi(\lambda_1)\geq\phi\left(2\cos\left(\frac{\pi}{n+1}\right)\right).$$
Suppose that equality in (\ref{teo14}) is maintained. Then in (\ref{EE2}) inequalities are changed by equalities. Thus, by (\ref{eq7}), we have  
\begin{equation}\label{imp14}
   |\lambda_1|=|\lambda_n|=2\cos{\left(\frac{\pi}{n+1}\right)}  \quad \text{and} \quad \lambda_2=\ldots=\lambda_{n-1}=0.
\end{equation}
From (\ref{imp14}), the adjacency matrix of $G$ it has a index of imprimitivity $h=2$ and similar to the proofs in Theorem \ref{T20}, one can easily obtain that $G \cong K_{p,q}$, with $p+q=n$. Further of (\ref{imp14}) and Lemma \ref{lema10}, we have $G \cong P_n$. Thus, the only paths that are complete bipartite are $G \cong P_2$ and $G \cong P_3.$
On the other hand, if $G \cong P_2$ or $G \cong P_3$ it is easy to check that the equality in (\ref{teo14}) holds.
\end{proof}


\bigskip


\end{document}